\documentclass[11pt]{article}
\usepackage{amsfonts,amssymb}
\usepackage{diagrams}

\newcounter{paragrafsubsub}[subsubsection]
\renewcommand{\theparagrafsubsub}{%
\thesubsubsection.\roman{paragrafsubsub}}
\newcommand{\paragrafsubsub}{%
\refstepcounter{paragrafsubsub}
{\bf \theparagrafsubsub}\hspace{0.2em}--- }

\newcounter{paragrafsub}[subsection]
\renewcommand{\theparagrafsub}{\thesubsection.\arabic{paragrafsub}}
\newcommand{\paragrafsub}{%
\refstepcounter{paragrafsub}
{\bf \theparagrafsub}\hspace{0.2em}--- }

\newcounter{paragraf}[section]
\renewcommand{\theparagraf}{\thesection.\arabic{paragraf}}
\newcommand{\paragraf}{%
\refstepcounter{paragraf}
{\bf \theparagraf}\hspace{0.2em}--- }

\newcommand\paragraphe{%
\par \indent
\ifcase\value{subsection} %
\paragraf
\else
\ifcase\value{subsubsection}\paragrafsub %
\else\paragrafsubsub
\fi\fi
}

\def\longto{\longrightarrow}

\def\RR{{\mathbb R}}\def\QQ{{\mathbb Q}}\def\ZZ{{\mathbb Z}}\def\NN{{\mathbb N}}
\def\CC{{\mathbb C}}

\def\Fl{{\mathcal Fl}}\def\cone{{\mathcal C}}\def\Face{{\mathcal F}}
\def\Gr{{\rm Gr}}\def\Gl{{\rm Gl}}
\def\Pic{\rm Pic}

\def\kbprod{\odot_0}

\def\Li{{\mathcal{L}}}

\newtheorem{prop}{Proposition}
\newtheorem{theo}{Theorem}
\newtheorem{coro}{Corollary}

\newenvironment{proof}{{\noindent\bf Proof.}}{\hfill $\square$}

\newenvironment{remark}{{\noindent\bf Remark.}}{}

\begin{document}
\title{A short geometric proof of a conjecture of Fulton}
\author{N. Ressayre}

\maketitle

\begin{abstract}
We give a new geometric proof of a conjecture of Fulton on the Littlewood-Richardson coefficients.
This conjecture was firstly proved by Knutson, Tao and Woodward using the Honeycomb theory.
A geometric proof was given by Belkale. Our proof is based on the geometry of Horn's cones.
\end{abstract}

\section{Introduction}
 
Recall that irreducible representations of $\Gl_r(\CC)$ are indexed by sequences 
$\lambda=(\lambda_1\geq\cdots\geq\lambda_r)\in\ZZ^r$. 
If $\lambda_r\geq 0$, $\lambda$ is called a {\it partition}.
Denote   the representation corresponding to $\lambda$ by $V_\lambda$.
Define Littlewood-Richardson coefficients $c_{\lambda\,\mu}^\nu\in\NN$ by:
$V_\lambda\otimes V_\mu=\sum_\nu c_{\lambda\,\mu}^\nu V_\nu$.
W.~Fulton conjectured that for any positive integer $N$,
$$
c_{\lambda\,\mu}^\nu=1 \ \Rightarrow\  c_{N\lambda\,N\mu}^{N\nu}=1.
$$
This conjecture was firstly proved by Knutson, Tao and Woodward \cite{KTW} using the Honeycomb theory.
A geometric proof was given by Belkale in \cite{belkale:geomHorn}.
The aim of this note is to give a short proof of this conjecture based on the geometry of Horn cones.

Note that the converse of Fulton's conjecture is a consequence of  Zelevinski's saturation conjecture.
This last conjecture was proved in \cite{KT:saturation,Bel:saturation,DW:saturation}.

The key observations of our proof are:
\begin{enumerate}
\item the non-trivial faces of codimension one of Horn cones corresponds to Littlewodd-Richardson 
coefficients equal to one;
\item each non zero Littlewood-Richardson coefficient give a linear inequality satisfied 
by Horn cones.
\end{enumerate}
Assume that  $c_{\lambda\,\mu}^\nu=1$. 
By Borel-Weyl's theorem, $c_{\lambda\,\mu}^\nu$ is the dimension of the $\Gl_r$ invariant sections
of a line bundle $\Li$ on a certain projective variety $X$.
Then, $c_{N\lambda\,N\mu}^{N\nu}$ is the dimension of the $\Gl_r$ invariant sections
of $\Li^{\otimes N}$. This implies that $c_{N\lambda\,N\mu}^{N\nu}\geq 1$.
In particular, $c_{N\lambda\,N\mu}^{N\nu}$ gives a linear inequality for a certain Horn cone: we 
prove that this inequality correspond to a face of codimension one. 

In this proof, the Littlewood-Richardson coefficients are mainly the coefficient structure of
the cohomology of the Grassmannians in the Schubert basis. Our technique can be applied to 
prove similar results for the coefficient structure of the Belkale-Kumar's product $\kbprod$ 
on the cohomology of others projective homogeneous spaces $G/P$. 

\section{Geometry of Horn cones}

\subsection{Horn's cone of Eigenvalues}

\paragraphe{\bf Schubert Calculus.}
Let $\Gr(a,b)$ be the Grassmann variety of $a$-dimensional subspaces $L$ of a 
fixed $a+b$-dimensional vector space $V$.
We fix a complete flag $F_\bullet$: $\{0\}=F_0\subset F_1\subset 
F_2\subset\cdots\subset F_{a+b}=V$.
For any subset $I=\{i_1<\cdots<i_a\}$ of cardinal $a$ in $\{1,\cdots,a+b\}$, there is a 
Schubert variety $\Omega_I(F_\bullet)$ in $\Gr(a,b)$ defined by
$$
\Omega_I(F_\bullet)=\{L\in\Gr(a,b)\,:\,
\dim(L\cap F_{i_j})\geq j {\rm\ for\ }1\leq j\leq n\}.
$$
The Poincar\'e dual of the homology class of $\Omega_I(F_\bullet)$ does not depend on $F_\bullet$;
it is denoted $\sigma_I$.
The $\sigma_I$ form a $\ZZ$-basis for the cohomology ring. 
It follows that for any subsets $I,\,J$ of cardinal $a$ in $\{1,\cdots,a+b\}$, there is a unique 
expression
$$
\sigma_I.\sigma_J=\sum_Kc_{IJ}^K\sigma_K,
$$
for integers $c_{IJ}^K$. We define $K^\vee$ by $i\in K^\vee$ if and only if $a+b+1-i\in K$. 
Then, if the sum of the codimensions of $\Omega_I(F_\bullet)$, $\Omega_J(F_\bullet)$ and 
$\Omega_I(F_\bullet)$ equals the dimension of $\Gr(a,b)$,  we have
$$
\sigma_I.\sigma_J.\sigma_K=c_{IJ}^{K^\vee}[{\rm pt}].
$$
\paragraphe{\bf Horn's cone.}
Let $H(n)$ denote the set of $n$ by $n$ Hermitian matrix.
For $A\in H(n)$, we denote its spectrum by $\alpha(A)=(\alpha_1,\cdots,\alpha_n)\in \RR^n$ 
repeated according to multiplicity and ordered such that $\alpha_1\geq\cdots\geq\alpha_n$. We set
$$
\Delta(n):=\{(\alpha(A),\alpha(B),\alpha(C))\in\RR^{3n}\,:\,
A,\,B,\,C\in H(n)  {\rm\ s.t.\ }A+B+C=0\}.
$$
Set $E(n)=\RR^{3n}$, let $E(n)^+$ denote the set of $(\alpha_i,\beta_i,\gamma_i)\in E(n)$
such that 
$\alpha_i\geq\alpha_{i+1},\, \beta_i\geq\beta_{i+1}$ and
$\gamma_i\geq\gamma_{i+1}$ for all $i=1,\cdots,n-1$. 
Let $E(n)^{++}$ denote the interior of $E(n)^+$.
Let $E_0(n)$ denote the hyperplane of points $(\alpha_i,\beta_i,\gamma_i)\in E(n)$ such that
$\sum\alpha_i+\sum\beta_i+\sum\gamma_i=0$. 
The set $\Delta(n)$ is a closed convex cone contained in $E_0(n)$ and of non empty interior in this
hyperplane.\\

\paragraphe\label{par:GITcone}{\bf GIT-cone}
Let $V$ be a complex $n$-dimensional vector space. 
Let $\Fl(V)$ denote the variety of complete flags of $V$. 
The group $G=\Gl(V)$ acts diagonaly on the variety $X=\Fl(V)^3$.
Let us fix a basis in $V$, $F_\bullet\in\Fl(V)$ the standard flag for this base, 
$B$ its stabilizer in $G$ and $T$ the torus consisting of diagonal matrices.
We identify the character groups $X(T)$ and $X(B)$ with $\ZZ^n$ in canonical way.
The line $\CC$ endowed by the action of $B^3$ given by $(\lambda,\,\mu,\,\nu)\in(\ZZ^n)^3=X(B^3)$
is denoted by $\CC_{(\lambda,\,\mu,\,\nu)}$.
The fiber product $G^3\times_{B^3}\CC_{(\lambda,\,\mu,\,\nu)}$ is a $G^3$-linearized line bundle
$\Li_{\lambda,\mu,\nu}$
on $X$; we denote by $\overline{\Li}_{\lambda,\,\mu,\,\nu}$ the $G$-linearized line bundle obtained 
by restricting the $G^3$-action to the diagonal.

We denote by $\cone^G(X)$ the rational cone generated by triples of partitions 
$(\lambda,\,\mu,\,\nu)$ such that $\Li_{\lambda,\,\mu,\,\nu}$ has non zero 
$G$-invariant sections. The fist proof of the following is due to Heckman \cite{Heck}, 
(see also \cite{Fulton:survey}.

\begin{theo}
\label{th:GITeigen}
  The cone $\Delta(n)$ is the closure of the rational convex cone $\cone^G(X)$.
\end{theo}

\subsection{Faces of $\Delta(n)$}

\paragraphe
We have a complete description of the linear forms on $E(n)$ which define faces of codimension one 
of $\Delta(n)$.
The first proof using Honeycombs is due to Knutson,Tao and Woodward (see \cite{KTW}).
A geometric proof is due to Belkale (\cite{Bel:irred}). 
In \cite{GITEigen}, I made a different geometric proof. 
A proof using quivers is also given in \cite{DW:comb}.   

\begin{theo}
\label{th:face}
  The hyperplanes $\alpha_i=\alpha_{i+1}$, $\beta_i=\beta_{i+1}$ and $\gamma_i=\gamma_{i+1}$
spanned by the codimension one faces of $E(n)^+$ intersects $\Delta(n)$ along faces of codimension 
one.

For any subsets $I,J$ and $K$ of $\{1,\cdots,n\}$ of the same cardinality such that 
$c_{IJ}^{K^\vee}=1$, the hyperplane 
$\sum_{i\in I}\alpha_i+\sum_{j\in J}\beta_j+\sum_{k\in K}\gamma_k=0$. 
intersects $\Delta(n)$ along a face $\Face_{IJK}$ of codimension one. 
Any face of codimension one intersecting $E(n)^{++}$ is obtain is this way.
\end{theo}

It is well known that if $c_{IJ}^{K^\vee}\neq 0$, 
for all $(\alpha,\beta,\gamma)\in\Delta(n)$, we have
$\sum_I\alpha_i+\sum_J\beta_j+\sum_K\gamma_k\leq 0$. In particular, the interstion 
betwenn $\Delta(n)$ and $\sum_I\alpha_i+\sum_J\beta_j+\sum_K\gamma_k=0$ is a face $\Face_{IJK}$
of $\Delta(n)$.

\paragraphe
We now review some notions of \cite{GITEigen,GITEigen2} and use notation of 
Paragraph~\ref{par:GITcone}.
Let $I,J$ and $K$ be  three subsets of $\{1,\cdots,n\}$ of the same cardinality $r$ 
such that $c_{IJ}^{K^\vee}\neq 0$.
We associate to this situation a pair $(C,\lambda)$ where $\lambda$ is a one parameter subgroup of 
$G$, and $C$ is an irreducible component of the set of fix points of $\lambda$ in $X$.
Consider the set $C^+$ of the $x\in X$ such that $\lim_{t\to 0}\lambda(t)x\in C$, and the parabolic
subgroup $P(\lambda)$ of $G$ associated to $\lambda$.
The assumption  $c_{IJ}^{K^\vee}\neq 0$ implies that the morphism
$$
\eta_{IJK}\,:\,G\times_{P(\lambda)}C^+\longto X,[g:x]\longmapsto g.x,
$$
is dominant with finite general fibers.
Now,  $\Face_{IJK}$ correspond to a face $\Face_{IJK}^\QQ$ of  $\cone^G(X)$:
the entire points in $\Face_{IJK}^\QQ$ correspond to the line bundles $\Li$ in $\cone^G(X)$ such that
$\lambda$ act trivialy on $\Li_{|C}$. 
By \cite[]{GITEigen} or \cite[]{GITEigen2},  the entire points in $\Face_{IJK}^\QQ$ correspond to the $G$-linearized 
line bundles $\Li$ on $X$ such that $X^{\rm ss}(\Li)$ intersects $C$.

By construction, $\lambda$ acts with two weights on $V$, the first one has multiplicity $r$ and 
the other one $n-r$. In particular, 
the centralizer $G^\lambda$ in $G$ of $\lambda$ is isomorphic to $\Gl_r\times \Gl_{n-r}$. 
Moreover, $C$ is isomorphic to $\Fl(\CC^r)^3\times\Fl(\CC^{n-r})^3$.
Now, consider the restriction morphism
$$
\rho_{IJK}^\QQ\,:\,\Pic^{G^3}(X)\longto \Pic^{(G^\lambda)^3}(C).
$$

\paragraphe
Let $I,J$ and $K$ be three subsets of $\{1,\cdots,n\}$ of the same cardinality $r$.
Define the linear isomorphism $\rho_{IJK}$ by:
$$
\begin{array}{ccc}
E(n)&\longto&E(r)\oplus E(n-r)\\
(\alpha_i,\beta_i,\gamma_i)&\longmapsto&
((\alpha_i)_{i\in I},(\beta_i)_{i\in J},(\gamma_i)_{i\in K})+
((\alpha_i)_{i\notin I},(\beta_i)_{i\notin J},(\gamma_i)_{i\notin K}).
\end{array}
$$
One easily checks that with evident identifications, $\rho_{IJK}$ is obtained from $\rho_{IJK}^\QQ$ 
by extending the scalar to the real numbers.

\begin{prop}
\label{prop:facerho}
Let $I,J$ and $K$ be as above with $c_{IJ}^{K^\vee}\neq 0$.
  Let $(\alpha,\beta,\gamma)\in E(n)^+$.
Then, $(\alpha,\beta,\gamma)\in \Face_{IJK}$ if and only if 
$\rho_{IJK}(\alpha,\beta,\gamma)\in \Delta(r)\times\Delta(n-r)$.
\end{prop}

\begin{proof}
Assume that   $\rho_{IJK}(\alpha,\beta,\gamma)\in \Delta(r)\times\Delta(n-r)$.
Let $A',B',C'\in H(r)$ and  $A'',B'',C''\in H(n-r)$ such that $A'+B'+C'=0$ and $A''+B''+C''=0$
whose spectrums correspond to $\rho_{IJK}(\alpha,\beta,\gamma)$.
Consider the three following matrices of $H(n)$
$$
A=\left(
  \begin{array}{cc}
    A' & 0\\0&A''
  \end{array}
\right),\ \
 B=\left(
  \begin{array}{cc}
    B' & 0\\0&B''
  \end{array}
\right),\ \
 C=\left(
  \begin{array}{cc}
    C' & 0\\0&C''
  \end{array}
\right).
$$
By construction, $\alpha$ is the spectrum of $A$ and $\sum_I\alpha_i={\rm tr}(A')$, and
similarly for $B$ and $C$. We deduce that $(\alpha,\beta,\gamma)\in \Face_{IJK}$.\\

By Theorem~\ref{th:GITeigen}, we can prove the converse for the cone $\cone^G(X)$.
Let $\Li\in \Face_{IJK}$. Since $X^{\rm ss}(\Li)$ intersects $C$, $C$ contains semistable points
for the action of $G^\lambda$ and $\rho_{IJK}^\QQ(\Li)$. 
It follows that  $\rho_{IJK}^\QQ(\Li)\in \Delta(r)\times\Delta(n-r)$.
\end{proof}

\begin{coro}
  \label{cor:faceint}
Let $I,J$ and $K$ be as in the proposition.
Then, if $\Face_{IJK}$ intersects $E(n)^{++}$, it has codimension one. 
In particular, $c_{IJ}^{K^\vee}=1$.
\end{coro}

\begin{proof}
By Proposition~\ref{prop:facerho},  $\Face_{IJK}\cap E(n)^{++}$ is isomorphic to an open subset
of $\Delta(r)\times\Delta(n-r)$. 
So,  $\Face_{IJK}$ has codimension 2 in $E(n)$ and so codimension one in $\Delta(n)$.
Now, Theorem~\ref{th:face} implies that  $c_{IJ}^{K^\vee}=1$. 
\end{proof}\\

\begin{remark}
Corollary~\ref{cor:faceint} for $\cone^G(X)$ is proved in \cite{GITEigen} by purely 
Geometric Invariant Theoretic methods; that is, without using Theorem~\ref{th:GITeigen}.

The first example of face $\Face_{IJK}$ with $c_{IJ}^{K^\vee}>1$ is obtained for $n=6$.
etc...
\end{remark}

\section{Proof of Fulton's conjecture}

Let $\lambda,\,\mu$ and $\nu$ be three partitions (with $r$ parts) such that $c_{\lambda\,\mu}^\nu=1$.
Let us fix $n$ such that $n-a$ is greater or equal to $\lambda_1,\,\mu_1$ and $\nu_1$.
Set $I=\{n-a+i-\lambda_i\,:\,i=1,\cdots,a\}\subset\{1,\cdots,n\}$. 
Similarly, we associate $J$ and $K$ to $\mu$ and $\nu$. 
It is well known that:
$$
c_{\lambda\,\mu}^\nu=c_{IJ}^K.
$$
By Theorem~\ref{th:face}, $\Face_{IJK^\vee}$ is a face of codimension one in $\Delta(n)$.

Let $(A,B,C,A',B',C')\in H(r)^3\times H(n-r)^3$ corresponding to a point in the relative
interior $\rho_{IJK^\vee}(\Face_{IJK^\vee})$. 
Consider $N$ generic perturbations $(A'_i,B'_i,C_i')$ of $(A',B',C')\in\Delta(n-r)$.
Consider now the Hermitian matrix $A''$ of size $r+N(n-r)$ diagonal by bloc with blocs
$A,\,A_1',\cdots,A_N'$; and similarly $B''$ and $C''$.

Let now, $I'',J''$ and $K''$ be the three subsets of $r+N(n-r)$ of cardinal $r$ corresponding 
to $N\lambda,\,N\mu$ and $N\nu$ respectively. 
It is clear that the image by $\rho_{I''J''K''^\vee}$ of the sprectrum of $(A'',B'',C'')$ belongs to
$\Delta(r)\times \Delta(N(n-r))$. By genericity of the matrices $A_i'$, $B_i'$ and
$C_i'$, this implies that $\Face_{I''J''K''^\vee}$ intersects $E(r+N(n-r))^{++}$.
Now, Corollary~\ref{cor:faceint} allows to conclude.

\bibliographystyle{amsalpha}
\bibliography{biblio}

\begin{center}
  -\hspace{1em}$\diamondsuit$\hspace{1em}-
\end{center}

\vspace{5mm}
\begin{flushleft}
N. R.\\
Universit{\'e} Montpellier II\\
D{\'e}partement de Math{\'e}matiques\\
Case courrier 051-Place Eug{\`e}ne Bataillon\\
34095 Montpellier Cedex 5\\
France\\
e-mail:~{\tt ressayre@math.univ-montp2.fr}  
\end{flushleft}

\end{document}